\documentclass[12pt]{amsart}
\usepackage{amsmath,amsthm,amsfonts,amssymb,mathrsfs}
\usepackage{color}
\date{\today}

\usepackage{hyperref}
\input{xy}
 \xyoption{all}
 \xyoption{arc}

\newtheorem{theorem}{Theorem}

\newtheorem{proposition}{Proposition}
\newtheorem{corollary}{Corollary}

\newtheorem{lemma}{Lemma}
\theoremstyle{definition}
\newtheorem{definition}{Definition}

\begin{document}

\title[On feebly compact semitopological semilattice $\exp_n\lambda$]{On feebly compact semitopological semilattice $\exp_n\lambda$}
\author{Oleg Gutik and Oleksandra Sobol}
\address{Faculty of Mechanics and Mathematics,
National University of Lviv, Universytetska 1, Lviv, 79000, Ukraine}
\email{oleg.gutik@lnu.edu.ua, ovgutik@yahoo.com, o.yu.sobol@gmail.com}



\begin{abstract}
We study feebly compact shift-continous topologies on the semilattice $\left(\exp_n\lambda,\cap\right)$. It is proved that such $T_1$-topology  is sequentially pracompact if and only if it is $\mathfrak{D}(\omega)$-compact.
\end{abstract}

\maketitle



We shall follow the terminology of~\cite{Carruth-Hildebrant-Koch-1983-1986, Engelking-1989, Gierz-Hofmann-Keimel-Lawson-Mislove-Scott-2003, Ruppert-1984}. If $X$ is a topological space and $A\subseteq X$, then by $\operatorname{cl}_X(A)$ and $\operatorname{int}_X(A)$ we denote the closure and the interior of $A$ in $X$, respectively. By $\omega$ we denote the first infinite cardinal and by $\mathbb{N}$ the set of positive integers. By $\mathfrak{D}(\omega)$ and $\mathbb{R}$ we denote an infinite countable discrete space and the real numbers with the usual topology, respectively.

\smallskip

A subset $A$ of a topological space $X$ is called \emph{regular open} if $\operatorname{int}_X(\operatorname{cl}_X(A))=A$.

\smallskip

We recall that a topological space $X$ is said to be
\begin{itemize}
  \item \emph{semiregular} if $X$ has a base consisting of regular open subsets;
  \item \emph{compact} if each open cover of $X$ has a finite subcover;
  \item \emph{sequentially compact} if each sequence $\{x_n\}_{n\in\mathbb{N}}$ of $X$ has a convergent subsequence in $X$;
  \item \emph{countably compact} if each open countable cover of $X$ has a finite subcover;
  \item \emph{$H$-closed} if $X$ is a closed subspace of every Hausdorff topological space in which it contained;
  \item \emph{infra H-closed} provided that any continuous image of $X$ into any first countable Hausdorff space is closed (see \cite{Hajek-Todd-1975});
  \item \emph{totally countably pracompact} if there exists a dense subset $D$ of the space $X$ such that each sequence of points of the set $D$ has a subsequence with the compact closure in $X$;
  \item \emph{sequentially pracompact} if there exists a dense subset $D$ of the space $X$ such that each sequence of points of the set $D$ has a convergent subsequence \cite{Gutik-Ravsky-20??};
  \item \emph{countably compact at a subset} $A\subseteq X$ if every infinite subset $B\subseteq A$  has  an  accumulation  point $x$ in $X$;
  \item \emph{countably pracompact} if there exists a dense subset $A$ in $X$  such that $X$ is countably compact at $A$;
  \item \emph{selectively sequentially feebly compact} if for every family $\{U_n\colon n\in \mathbb{N}\}$ of non-empty open subsets of $X$, one can choose a point $x_n\in U_n$ for every $n\in \mathbb{N}$ in such a way that the sequence $\{x_n\colon n\in \mathbb{N}\}$ has a convergent subsequence (\cite{Dorantes-Aldama-Shakhmatov-2017});
  \item \emph{sequentially feebly compact} if for every family $\{U_n\colon n\in \mathbb{N}\}$ of non-empty open subsets of $X$, there exists an infinite set $J\subseteq \mathbb{N}$ and a point $x\in X$ such that the set $\{n\in J\colon W\cap U_n=\varnothing\}$ is finite for every open neighborhood $W$ of $x$ (see \cite{Dow-Porter-Stephenson-Woods-2004});
  \item \emph{selectively feebly compact} for each sequence $\{U_n\colon n\in \mathbb{N}\}$ of non-empty open subsets of $X$, one can choose a point $x\in X$ and a point $x_n\in U_n$ for each $n\in \mathbb{N}$ such that the set $\{n\in \mathbb{N}\colon x_n\in W\}$ is infinite for every open neighborhood $W$ of $x$ (\cite{Dorantes-Aldama-Shakhmatov-2017});
  \item \emph{feebly compact} (or \emph{lightly compact}) if each locally finite open cover of $X$ is finite~\cite{Bagley-Connell-McKnight-Jr-1958};
  \item $d$-\emph{feebly compact} (or \emph{\textsf{DFCC}}) if every discrete family of open subsets in $X$ is finite (see \cite{Matveev-1998});
  \item \emph{pseudocompact} if $X$ is Tychonoff and each continuous real-valued function on $X$ is bounded;
  \item $Y$-\emph{compact} for some topological space $Y$, if $f(X)$ is compact for any continuous map $f\colon X\to Y$.
\end{itemize}

The following diagram describes relations between the above defined classes of topological spaces.

\begin{equation*}
\xymatrix{
*+[F]{\begin{array}{c}
                                                          \hbox{\textbf{\small{sequentially}}}\\
                                                          \hbox{\textbf{\small{compact}}}
                                                        \end{array}}\ar@/_1pc/[dd]\ar@/^1pc/[rd] &*+[F]{\begin{array}{c}
                                                          \hbox{\textbf{\small{compact}}}
                                                        \end{array}
}\ar[d]\ar@/^2pc/[dr] & \\
 &*+[F]{\begin{array}{c}
                                                          \hbox{\textbf{\small{countably}}}\\
                                                          \hbox{\textbf{\small{compact}}}
                                                        \end{array}}\ar[d]
                                                        \ar@/^4pc/[lu]|-{\begin{array}{c}
                                                          \small{T_3\hbox{-space}}\\
                                                          \small{\hbox{+scattered}}
                                                        \end{array}} \ar@/^1pc/[lu]|-{\hbox{\small{sequential}}} &
 *+[F]{\begin{array}{c}
                                                          \hbox{\textbf{\small{totally}}}\\
                                                          \hbox{\textbf{\small{countably}}}\\
                                                          \hbox{\textbf{\small{pracompact}}}
                                                        \end{array}}\ar@/^2pc/[dl]\\
*+[F]{\begin{array}{c}
                                                          \hbox{\textbf{\small{sequentially}}}\\
                                                          \hbox{\textbf{\small{pracompact}}}
                                                        \end{array}}\ar[d]\ar[r]
 &*+[F]{\begin{array}{c}
                                                          \hbox{\textbf{\small{countably}}}\\
                                                          \hbox{\textbf{\small{pracompact}}}
                                                        \end{array}}\ar@/_5.5pc/[dd] &  \\
*+[F]{\begin{array}{c}
                                                          \hbox{\textbf{\small{selectively}}}\\
                                                          \hbox{\textbf{\small{sequentially}}}\\
                                                          \hbox{\textbf{\small{feebly compact}}}
                                                        \end{array}}\ar[r]\ar[d]&
*+[F]{\begin{array}{c}
                                                          \hbox{\textbf{\small{sequentially}}}\\
                                                          \hbox{\textbf{\small{feebly compact}}}
                                                        \end{array}} & *+[F]{\hbox{\textbf{\small{$d$-feebly compact}}}} \ar@/_-1.9pc/[ld]|-{\small{\hbox{quasi-regular}}} &
                                                        *+[F]{\hbox{\textbf{\small{H-closed}}}}\ar\ar@/_-2.2pc/[lld]\ar@/_8pc/[uuull]|-{\small{\hbox{regular}}}\\
*+[F]{\begin{array}{c}
                                                          \hbox{\textbf{\small{selectively}}}\\
                                                          \hbox{\textbf{\small{feebly compact}}}
                                                        \end{array}}\ar[r]                                                          &*+[F]{\hbox{\textbf{\small{feebly compact}}}}\ar@/_6.3pc/[uuu]|-{\small{\hbox{normal}}} \ar@/_-.2pc/[u]|-{\hbox{\footnotesize{Fr\'{e}chet-Urysohn}}}\ar@/_.2pc/[ur]\ar[d] 
                                                        & &\\
&*+[F]{\hbox{\textbf{\small{infra H-closed}}}}\ar@/_.1pc/[rr]|-{\footnotesize{\hbox{Tychonoff}}}\ar[d]& & *+[F]{\hbox{\textbf{\small{pseudocompact}}}}\ar@/_.2pc/[ull]\\
&*+[F]{\hbox{\textbf{\small{$\mathbb{R}$-compact}}}}\ar[d] &&\\
&*+[F]{\hbox{\textbf{\small{$\mathfrak{D}(\omega)$-compact}}}} &&\\
&&&
}
\end{equation*}


A \emph{semilattice} is a commutative semigroup of idempotents. On  a semilattice $S$ there exists a natural partial order: $e\leqslant f$
\emph{if and only if} $ef=fe=e$. For any element $e$ of a semilattice $S$ we put
\begin{equation*}
  {\uparrow}e=\left\{f\in S\colon e\leqslant f\right\}.
\end{equation*}

A {\it topological} ({\it semitopological}) {\it semilattice} is a topological space together with a continuous (separately continuous) semilattice operation. If $S$ is a~semilattice and $\tau$ is a topology on $S$ such that $(S,\tau)$ is a topological semilattice, then we shall call $\tau$ a \emph{semilattice} \emph{topology} on $S$, and if $\tau$ is a topology on $S$ such that $(S,\tau)$ is a semitopological semilattice, then we shall call $\tau$ a \emph{shift-continuous} \emph{topology} on~$S$.

For an arbitrary positive integer $n$ and an arbitrary non-zero cardinal $\lambda$ we put
\begin{equation*}
  \exp_n\lambda=\left\{A\subseteq \lambda\colon |A|\leqslant n\right\}.
\end{equation*}

It is obvious that for any positive integer $n$ and any non-zero cardinal $\lambda$ the set $\exp_n\lambda$ with the binary operation $\cap$ is a semilattice. Later in this paper by $\exp_n\lambda$ we shall denote the semilattice $\left(\exp_n\lambda,\cap\right)$.

This paper is a continuation of \cite{Gutik-Sobol-2016} and \cite{Gutik-Sobol-2016a}. In \cite{Gutik-Sobol-2016} we studied feebly compact semitopological semilattices $\exp_n\lambda$. Therein, all compact semilattice $T_1$-topologies on $\exp_n\lambda$ were described. In \cite{Gutik-Sobol-2016} it was proved that for an arbitrary positive integer $n$ and an arbitrary infinite cardinal $\lambda$ every $T_1$-semitopological countably compact semilattice $\exp_n\lambda$ is a compact topological semilattice. Also, there  we constructed a countably pracompact $H$-closed quasiregular non-semiregular topology $\tau_{\operatorname{\textsf{fc}}}^2$ such that $\left(\exp_2\lambda,\tau_{\operatorname{\textsf{fc}}}^2\right)$ is a semitopological semilattice with the discontinuous semilattice operation and show that for an arbitrary positive integer $n$ and an arbitrary infinite cardinal $\lambda$ a semiregular feebly compact semitopological semilattice $\exp_n\lambda$ is a compact topological semilattice. In \cite{Gutik-Sobol-2016a} we proved that for any shift-continuous $T_1$-topology $\tau$ on $\exp_n\lambda$ the following conditions are equivalent:
$(i)$ $\tau$ is countably pracompact; $(ii)$ $\tau$ is feebly compact; $(iii)$ $\tau$ is $d$-feebly compact; $(iv)$~$\left(\exp_n\lambda,\tau\right)$ is an $H$-closed space.

In \cite{Artico-Marconi-Pelant-Rotter-Tkachenko-2002} was proved that \emph{every pseudocompact topological group is sequentially feebly compact}. Also, by Corollary~4.6 of \cite{Dorantes-Aldama-Shakhmatov-2017}, the Cantor cube $D^\mathfrak{c}$ is selectively sequentially feebly compact.
By \cite[Theorem 3.10.33]{Engelking-1989}, $D^\mathfrak{c}$ is not sequentially compact. Therefore, the compact topological group $G=D^\mathfrak{c}$ is selectively sequentially feebly compact but not sequentially feebly compact. Also, there exists a dense subgroup of $\mathbb{Z}^\mathfrak{c}_2$, where $\mathbb{Z}^\mathfrak{c}_2$ is the $\mathfrak{c}$-power of the cyclic two-elements group, which is selectively pseudocompact  but not selectively sequentially pseudocompact \cite{Shakhmatov-Yanez-2017arxiv}.
This and our above results of \cite{Gutik-Sobol-2016} and \cite{Gutik-Sobol-2016a} motivates us to investigate selective (sequential) feeble compactness of the semilattice $\exp_n\lambda$ as a semitopological semigroup.

Namely, we show that a shift-continuous $T_1$-semitopological semilattice $\exp_n\lambda$ is sequentially countably pracompact if and only if it is $\mathfrak{D}(\omega)$-compact.

\begin{lemma}\label{lemma-1}
Let $n$ be any positive integer and $\lambda$ be any infinite cardinal. Then the set  of isolated points of a $T_1$-semitopological semilattice $\exp_n\lambda$ is dense in it.
\end{lemma}

\begin{proof}
Fix an arbitrary non-empty open subset $U$ of $\exp_n\lambda$. There exists $y\in\exp_n\lambda$ such that ${\uparrow}y\cap U=\{y\}$. By Proposition~1$(iii)$ from \cite{Gutik-Sobol-2016}, ${\uparrow}y$ is an open-and-closed subset of $\exp_n\lambda$ and hence $y$ is an isolated point in $\exp_n\lambda$.
\end{proof}

A family of non-empty sets $\{A_i\colon i\in \mathscr{I}\}$  is called a \emph{$\Delta$-system} (a \emph{sunflower} or a \emph{$\Delta$-family}) if the pairwise intersections of the members are the same, i.e., $A_i\cap A_j=S$ for some set $S$ (for $i\neq j$ in $\mathscr{I}$) \cite{Komjath-Totik-2006}.
The following statement is well known as the \emph{Sunflower Lemma} or the \emph{Lemma about a $\Delta$-system} (see \cite[p. 107]{Komjath-Totik-2006}).

\begin{lemma}\label{Subflower_Lemma}
Every infinite family of $n$-element sets $(n<\omega)$ contains an infinite $\Delta$-subfamily.
\end{lemma}

\begin{proposition}\label{proposition-2}
Let $n$ be any positive integer and $\lambda$ be any infinite cardinal. Then every
feebly compact $T_1$-semitopological semilattice $\exp_n\lambda$ is sequentially pracompact.
\end{proposition}

\begin{proof}
Suppose to the contrary that there exists a feebly compact $T_1$-semitopological semilattice $\exp_n\lambda$ which is not sequentially  pracompact. Then every dense subset $D$ of $\exp_n\lambda$ contains a sequence of points from $D$ which has no  a convergent subsequence.

By Proposition~1 of \cite{Gutik-Sobol-2016a} the subset $\exp_n\lambda\setminus\exp_{n-1}\lambda$ is dense in $\exp_n\lambda$ and by Proposition~1$(ii)$ of \cite{Gutik-Sobol-2016} every point of the set $\exp_n\lambda\setminus\exp_{n-1}\lambda$ is isolated in $\exp_n\lambda$. Then the set $\exp_n\lambda\setminus\exp_{n-1}\lambda$ contains an infinite sequence of points $\{x_p\colon p\in \mathbb{N}\}$ which has not a convergent subsequence. By Lemma~\ref{Subflower_Lemma} the sequence $\{x_p\colon p\in \mathbb{N}\}$ contains an infinite $\Delta$-subfamily, that is an infinite subsequence $\{x_{p_i}\colon i\in \mathbb{N}\}$ such that there exists $x\in \exp_n\lambda$ such that $x_{p_i}\cap x_{p_j}=x$ for any distinct $i,j\in\mathbb{N}$.

Suppose that $x=0$ is the zero of the semilattice $\exp_n\lambda$. Since the sequence $\{x_{p_i}\colon i\in \mathbb{N}\}$ is an infinite $\Delta$-subfamily, the intersection $\{x_{p_i}\colon i\in \mathbb{N}\}\cap{\uparrow}y$ contains at most one set for every non-zero element $y\in \exp_n\lambda$. Thus $\exp_n\lambda$ contains an infinite locally finite family of open non-empty subsets which contradicts the feeble compactness of $\exp_n\lambda$.

If $x$ is a non-zero element of the semilattice $\exp_n\lambda$ then by Proposition~1$(ii)$ of \cite{Gutik-Sobol-2016}, ${\uparrow}x$ is an open-and-closed subspace of $\exp_n\lambda$, and hence by Theorem~14 from \cite{Bagley-Connell-McKnight-Jr-1958} the space ${\uparrow}x$ is feebly compact. We observe that $x$ is zero of the semilattice ${\uparrow}x$, which contradicts so similarly the previous part of the proof.
We obtain a contradiction.
\end{proof}

\begin{proposition}\label{proposition-3}
Let $n$ be an arbitrary positive integer and $\lambda$ be an arbitrary infinite cardinal. Then every
feebly compact $T_1$-semitopological semilattice $\exp_n\lambda$ is totally countably pracompact.
\end{proposition}

\begin{proof}
We put $D=\exp_n\lambda\setminus\exp_{n-1}\lambda$. By Proposition~1 of \cite{Gutik-Sobol-2016a} the subset $D$ is dense in $\exp_n\lambda$ and by Proposition~1$(ii)$ of \cite{Gutik-Sobol-2016} every point of the set $D$ is isolated in $\exp_n\lambda$. Fix an arbitrary sequence $\{x_p\colon p\in \mathbb{N}\}$ of points of $D$. By Lemma~\ref{Subflower_Lemma} the sequence $\{x_p\colon p\in \mathbb{N}\}$ contains an infinite $\Delta$-subfamily.

Suppose that $x=0$ is the zero of the semilattice $\exp_n\lambda$. Since the sequence $\{x_{p_i}\colon i\in \mathbb{N}\}$ is an infinite $\Delta$-subfamily, the intersection $\{x_{p_i}\colon i\in \mathbb{N}\}\cap{\uparrow}y$ contains at most one point of the sequence for every non-zero element $y\in \exp_n\lambda$. By Proposition~1$(ii)$ of \cite{Gutik-Sobol-2016} for every point $a\in \exp_n\lambda\setminus\{0\}$ there exists an open neighbourhood $U(a)$ of $a$ in $\exp_n\lambda$ such that $U(a)\subseteq {\uparrow}a$ and hence our assumption implies that zero $0$ is a unique accumulation point of the sequence $\{x_{p_i}\colon i\in \mathbb{N}\}$. Since by Lemma~1 from \cite{Gutik-Sobol-2016} for an arbitrary open neighbourhood $W(0)$ of zero $0$ in $\exp_n\lambda$ there exist finitely many non-zero elements $y_1,\ldots, y_{k}\in \exp_n\lambda$ such that
\begin{equation*}
  \left(\exp_n\lambda\setminus\exp_{n-1}\lambda\right)\subseteq W(0)\cup {\uparrow}{y_1}\cup\cdots\cup {\uparrow}{y_{k}},
\end{equation*}
we get that $\operatorname{cl}_{\exp_n\lambda}(\{x_{p_i}\colon i\in \mathbb{N}\})=\{0\}\cup\{x_{p_i}\colon i\in \mathbb{N}\}$ is a compact subset of $\exp_n\lambda$.

If $x$ is a non-zero element of the semilattice $\exp_n\lambda$ then by Proposition~1$(ii)$ of \cite{Gutik-Sobol-2016}, ${\uparrow}x$ is an open-and-closed subspace of $\exp_n\lambda$, and hence by Theorem~14 of \cite{Bagley-Connell-McKnight-Jr-1958} the space ${\uparrow}x$  is feebly compact. Then $x$ is zero of the semilattice ${\uparrow}x$ and by the previous part of the proof we have that $\operatorname{cl}_{\exp_n\lambda}(\{x_{p_i}\colon i\in \mathbb{N}\})=\{x\}\cup\{x_{p_i}\colon i\in \mathbb{N}\}$ is a compact subset of $\exp_n\lambda$.
\end{proof}

\begin{proposition}\label{proposition-4}
Let $n$ be any positive integer and $\lambda$ be any infinite cardinal. Then every $\mathfrak{D}(\omega)$-compact $T_1$-semitopological semilattice $\exp_n\lambda$ is feebly compact.
\end{proposition}

\begin{proof}
Suppose to the contrary that there exists a $\mathfrak{D}(\omega)$-compact $T_1$-semitopological semilattice $\exp_n\lambda$ which is not feebly compact. Then there exists an infinite locally finite family $\mathscr{U}=\{U_i\}$ of open non-empty subsets of $\exp_n\lambda$. Without loss of generality we may assume that the family $\mathscr{U}=\{U_i\}$ is countable, i.e., $\mathscr{U}=\{U_i\colon i\in\mathbb{N}\}$. Lemma~\ref{lemma-1} implies that for every $U_i\in\mathscr{U}$ there exists $a_i\in U_i$ such that  $\mathscr{U}^*=\{\{a_i\}\colon i\in\mathbb{N}\}$ is a family of isolated points of $\exp_n\lambda$. Since the family $\mathscr{U}$ is locally finite, without loss of generality we may assume that $a_i\neq a_j$ for distinct $i,j\in\mathbb{N}$. The family $\mathscr{U}^*$ is locally finite as a refinement of a locally finite family $\mathscr{U}$.  Since  $\exp_n\lambda,\tau$ is a $T_1$-space, $\bigcup\mathscr{U}^*$ is a closed subset in $\exp_n\lambda$ and hence the map $f\colon \exp_n\lambda\to \mathbb{N}_{\mathfrak{d}}$, where $\mathbb{N}_{\mathfrak{d}}$ is the set of positive integers with the discrete topology, defined by the formula
\begin{equation*}
  f(b)=
\left\{
  \begin{array}{cl}
    1, & \hbox{if~} b\in\exp_n\lambda\setminus\bigcup\mathscr{U}^*;\\
    i+1, & \hbox{if~} b=a_i \hbox{~for some~} i\in\mathbb{N},
  \end{array}
\right.
\end{equation*}
is continuous. This contradicts $\mathfrak{D}(\omega)$-compactness of the space $\exp_n\lambda$, because every two infinite countable discrete spaces are homeomorphic.
\end{proof}

We summarise our results in the  following theorem.

\begin{theorem}\label{theorem-5}
Let $n$ be any positive integer and $\lambda$ be any infinite cardinal. Then for any $T_1$-semitopological semilattice $\exp_n\lambda$ the following conditions are equivalent:
\begin{itemize}
  \item[$(i)$] $\exp_n\lambda$ is sequentially pracompact;
  \item[$(ii)$] $\exp_n\lambda$ is totally countably pracompact;
  \item[$(iii)$] $\exp_n\lambda$ is feebly compact;
  \item[$(iv)$] $\exp_n\lambda$ is $\mathfrak{D}(\omega)$-compact.
\end{itemize}
\end{theorem}

\begin{proof}
Implications $(i)\Rightarrow(iii)$, $(ii)\Rightarrow(iii)$ and $(iii)\Rightarrow(iv)$ are trivial. The corresponding their converse implications $(iii)\Rightarrow(i)$, $(iii)\Rightarrow(ii)$ and $(iv)\Rightarrow(i ii)$ follow from Propositions~\ref{proposition-2}, \ref{proposition-3} and \ref{proposition-4}, respectively.
\end{proof}

It is well known that the (Tychonoff) product of pseudocompact spaces is not necessarily pseudocompact (see \cite[Section~3.10]{Engelking-1989}). On the other hand Comfort and Ross in \cite{Comfort-Ross-1966} proved that a Tychonoff product of an arbitrary family of pseudocompact topological groups is a pseudocompact topological group.  Ravsky in \cite{Ravsky-2003arxiv} generalized the Comfort--Ross Theorem and proved that a Tychonoff
product of an arbitrary non-empty family of feebly compact paratopological groups is feebly compact. Also, a counterpart of the Comfort--Ross Theorem for pseudocompact primitive topological inverse semigroups and  primitive inverse semiregular feebly compact semitopological semigroups with closed maximal subgroups were proved in \cite{Gutik-Pavlyk-2013} and \cite{Gutik-Ravsky-2015}, respectively.

Since a Tychonoff product of H-closed spaces is H-closed (see \cite[Theorem~3]{Chevalley-Frink-1941} or \cite[3.12.5~(d)]{Engelking-1989}) Theorem~\ref{theorem-5} implies a counterpart of the Comfort--Ross Theorem for feebly compact semitopological semilattices $\exp_n\lambda$:

\begin{corollary}\label{corollary-6}
Let $\left\{\exp_{n_i}\lambda_i\colon i\in\mathscr{I}\right\}$ be a family of non-empty feebly compact $T_1$-semitopological semilattices and $n_i\in\mathbb{N}$ for all $i\in\mathscr{I}$. Then the Tychonoff product $\prod\left\{\exp_{n_i}\lambda_i\colon i\in\mathscr{I}\right\}$ is feebly compact.
\end{corollary}

\begin{definition}\label{definition-7}
If $\left\{X_i\colon i\in \mathscr{I}\right\}$ is a family of sets, $X=\prod\left\{X_i\colon i\in \mathscr{I}\right\}$ is their Cartesian product and $p$ is a point in $X$,
then the subset
\begin{equation*}
  \Sigma(p,X)=\left\{ x \in X\colon \left| \left\{ i\in\mathscr{I}\colon x(i)\neq p(i)\right\}\right|\leqslant\omega\right\}
\end{equation*}
of $X$ is called the \emph{$\Sigma$-product} of $\left\{X_i\colon i\in \mathscr{I}\right\}$ with the basis point $p\in X$. In the case when $\left\{X_i\colon i\in \mathscr{I}\right\}$ is a family of topological spaces we assume that $\Sigma(p,X)$ is a subspace of the Tychonoff product $X=\prod\left\{X_i\colon i\in \mathscr{I}\right\}$.
\end{definition}

It is obvious that if $\left\{X_i\colon i\in \mathscr{I}\right\}$ is a family of semilattices then $X=\prod\left\{X_i\colon i\in \mathscr{I}\right\}$  is a semilattice as well. Moreover $\Sigma(p,X)$ is a subsemilattice of $X$ for arbitrary $p\in X$. Then Theorem~\ref{theorem-5} and Proposition~2.2 of \cite{Gutik-Ravsky-20??} imply the following corollary.

\begin{corollary}\label{corollary-8}
Let $\left\{\exp_{n_i}\lambda_i\colon i\in\mathscr{I}\right\}$  be a family of non-empty feebly compact $T_1$-semitopologi\-cal semilattices and $n_i\in\mathbb{N}$ for all $i\in\mathscr{I}$. Then for every point $p$ of the product $X=\prod\left\{\exp_{n_i}\lambda_i \colon i\in\mathscr{I}\right\}$ the $\Sigma$-product $\Sigma(p,X)$ is feebly compact.
\end{corollary}
\section*{Acknowledgements}

The author acknowledge Alex Ravsky and the referee for their comments and
suggestions.

\end{document}